\nonstopmode
\documentclass[10pt]{amsart}
\usepackage{graphicx}
\usepackage{latexsym}
\usepackage{fancyhdr}
\usepackage{amsmath, amssymb, amsthm}
\usepackage[all]{xy}
\usepackage{pdflscape}
\usepackage{longtable}
\usepackage{rotating}
\usepackage{verbatim}
\usepackage{hyperref}
\hypersetup{
    linkcolor = blue,
    citecolor = blue,
    urlcolor  = blue,
    colorlinks = true,
}

\usepackage{cleveref}
\usepackage{subfigure}
\usepackage{mathrsfs}
\usepackage{mdwlist}
\usepackage{dsfont}
\usepackage{mathtools}
\usepackage{float}
\usepackage{color}
\usepackage{stmaryrd}

\definecolor{teal}{rgb}{0.0, 0.5, 0.5}

\newcounter{mnotecount}[section]

\newcommand{\rmnote}[1]{}%{\mnote{#1}}

\allowdisplaybreaks

\DeclareFontFamily{U}{mathb}{\hyphenchar\font45}
\DeclareFontShape{U}{mathb}{m}{n}{
      <5> <6> <7> <8> <9> <10> gen * mathb
      <10.95> mathb10 <12> <14.4> <17.28> <20.74> <24.88> mathb12
      }{}
\DeclareSymbolFont{mathb}{U}{mathb}{m}{n}
\DeclareFontSubstitution{U}{mathb}{m}{n}

\let\dot\relax
\DeclareMathAccent{\dot}{0}{mathb}{"39}
\let\ddot\relax
\DeclareMathAccent{\ddot}{0}{mathb}{"3A}
\let\dddot\relax
\DeclareMathAccent{\dddot}{0}{mathb}{"3B}
\let\ddddot\relax
\DeclareMathAccent{\ddddot}{0}{mathb}{"3C}

\theoremstyle{plain}
\newtheorem*{theorem*}{Theorem}
\newtheorem{theorem}{Theorem}[section]
\newtheorem*{lemma*}{Lemma}
\newtheorem{lemma}[theorem]{Lemma}
\newtheorem*{assumption*}{Assumption}

\newtheorem*{proposition*}{Proposition}
\newtheorem{proposition}[theorem]{Proposition}
\newtheorem*{corollary*}{Corollary}
\newtheorem{corollary}[theorem]{Corollary}
\newtheorem*{claim*}{Claim}

\newtheorem*{conjecture*}{Conjecture}

\newtheorem*{question*}{Question}
\theoremstyle{definition}
\newtheorem*{definition*}{Definition}

\newtheorem*{example*}{Example}

\newtheorem*{algorithm*}{Algorithm}
\newtheorem*{remark*}{Remark}
\newtheorem*{remarks*}{Remarks}
\newtheorem{remark}[theorem]{Remark}

\newtheorem*{convention*}{Convention}

\theoremstyle{plain}

\Crefname{l}{Lemma}{Lemmas}    %declares the type of the environment for cleverref, use \label[<type>]{<label>}
\Crefname{p}{Proposition}{Propositions}
\Crefname{t}{Theorem}{Theorems}
\Crefname{c}{Corollary}{Corollaries}
\Crefname{r}{Remark}{Remarks}
\Crefname{d}{Definition}{Definitions}
\Crefname{e}{Example}{Examples}

\def\al{\alpha}
\def\be{\beta}

\def\de{\delta}
\def\ep{\epsilon}

\def\rh{\rho}

\def\si{\sigma}

\def\ta{\tau}

\def\vh{\varphi}

\def\om{\omega}

\def\Om{\Omega}

\def\N{\mathbb{N}}

\def\R{\mathbb{R}}

\def\cB{\mathcal{B}}

\def\cD{\mathcal{D}}
\def\cE{\mathcal{E}}

\def\cQ{\mathcal{Q}}

\def\cS{\mathcal{S}}

\def\fM{\mathfrak{M}}

\def\p{\partial}

\def\<{\langle}
\def\>{\rangle}

\def\ol{\overline}

\let\on=\operatorname

\newcommand{\sr}[1]%
{\ifmmode{}^\dagger\else${}^\dagger$\fi\ifvmode
\vbox to 0pt{\vss
 \hbox to 0pt{\hskip\hsize\hskip1em
 \vbox{\hsize3cm\raggedright\pretolerance10000
 \noindent #1\hfill}\hss}\vss}\else
 \vadjust{\vbox to0pt{\vss%
 \hbox to 0pt{\hskip\hsize\hskip1em%
 \vbox{\hsize3cm\raggedright\pretolerance10000%
 \noindent #1\hfill}\hss}\vss}}\fi%
}

\def\A{\;\forall}
\def\E{\;\exists}

\providecommand{\mapsfrom}{\kern.2em%
\setbox0=\hbox{$\leftarrow$\kern-.10em\rule[0.26mm]{0.1mm}{1.3mm}}\box0%
\kern.3em}

\title[Interpolation of derivatives and ultradifferentiable regularity]
{Interpolation of derivatives and ultradifferentiable regularity}

\author[A.~Rainer]{Armin Rainer}

\author[G.~Schindl]{Gerhard Schindl}

\address{Fakult\"at f\"ur Mathematik, Universit\"at Wien,
Oskar-Morgenstern-Platz~1, A-1090 Wien, Austria}
\email{armin.rainer@univie.ac.at}
\email{gerhard.schindl@univie.ac.at}

\begin{document}

\begin{abstract}
    Interpolation inequalities for $C^m$ functions allow to bound derivatives of intermediate order $0 < j<m$ by 
    bounds for the derivatives of order $0$ and $m$. 
    We review various interpolation inequalities for $L^p$-norms ($1 \le p \le \infty$) in arbitrary finite dimensions.
    They allow us to study ultradifferentiable regularity by lacunary estimates in a comprehensive way, 
    striving for minimal assumptions on the weights.
\end{abstract}

\thanks{This research was funded in whole or in part by the Austrian Science Fund (FWF) DOI 10.55776/P32905 and DOI 10.55776/P33417.
For open access purposes, the authors have applied a CC BY public copyright license to any author-accepted manuscript version arising from this submission.}
\keywords{Interpolation inequalities, Denjoy--Carleman classes, ultradifferentiable regularity, lacunary estimates}
\subjclass[2020]{
	26D10,      %Inequalities involving derivatives and differential and integral operators 
    26E05,      %Real analytic functions
    26E10,  	%$C^\infty$-functions, quasi-analytic functions
    41A17,  	%Inequalities in approximation (Bernstein, Jackson, Nikol'skiĭ-type inequalities)
    46E10}      %Topological linear spaces of continuous, differentiable or analytic functions
%\date{December 5, 2024}

\maketitle

\section{Introduction}

Ultradifferentiable classes are classes of $C^\infty$ functions defined by prescribed growth behavior of the 
infinite sequence of derivatives. The most classical among them are the Denjoy--Carleman classes, where the 
growth of the derivatives is dominated by a weight sequence $M=(M_j)_{j \ge 0}$. The Gevrey classes (in particular, 
the real analytic class),
which play an important role in the theory of differential equations,
are special cases thereof.

In applications, sometimes the question arises whether the ultradifferentiable regularity can be concluded 
if only lacunary information on the growth of the derivatives is available. 
For instance, Bolley, Camus, and M\'etivier \cite{Bolley:1991ab} studied the analyticity of analytic vectors 
and Liess \cite{Liess:1990aa} partially extended their work to Denjoy--Carleman classes.
Knowing suitable bounds for the sequence $(P^jf)_{j\ge0}$, where $P$ is an elliptic linear partial differential operator, 
one would like to conclude similar bounds for all derivatives $f$.
In \cite{Rainer:2019ab}, this approach was applied to isotropic functions and used to prove Chevalley-type results. 

The basic problem is the following.
Suppose that we know that a smooth function satisfies certain ultradifferentiable bounds  
for the derivatives of order $k_j$, where $(k_j)$ is a strictly increasing sequence of integers,
can we deduce that its derivatives of all orders satisfy the ultradifferentiable bounds? 
An affirmative answer clearly depends on conditions for the base sequence $(k_j)$, the weight $M=(M_j)$, 
and on suitable interpolation inequalities.

In the recent paper \cite{Albano:2023aa}, Albano and Mughetti, 
motivated by $L^2$ methods for proving local regularity of solutions for (degenerate) elliptic equations 
(cf.\ \cite{Bove:2017aa,Bove:2019aa,Bove:2020aa}),
gave sufficient conditions for Denjoy--Carleman classes of Roumieu type 
on a compact interval of $\R$, based on the Cartan--Gorny inequality (cf.\ \Cref{prop:CartanGorny}).
The base sequence $(k_j)$ is required to be such that $k_{j+1}/k_j$ is bounded. 
For the weight sequence, the authors assume a rather strong condition (which we recall and discuss in \Cref{rem:Albano}).
In fact, we prove a version under a weaker and more natural condition in \Cref{thm:CG} (see also \Cref{thm:DC}), namely, 
boundedness of $m_{k_{j+1}}/m_{k_j}$, where $m_j:=M_j^{1/j}$.
Note that Liess showed (in \cite{Liess:1990aa}) that, under certain assumptions on $M$, boundedness of $m_{k_{j+1}}/m_{k_j}$ is also a necessary condition 
for the interpolation problem (see \Cref{rem:Liess}).
We discuss at the end of \Cref{sec:compare} necessity of the conditions and how the different conditions are related.

The main purpose of this paper is 
to treat the interpolation problem in a broad and comprehensive way, striving for minimal assumptions on the weights. 
To this end, we review various interpolation inequalities for $L^p$-norms, where $1 \le p \le \infty$,
and work in arbitrary (finite) dimensions. Note that the growth of the involved constants is crucial in the ultradifferentiable setting.
We consider ultradifferentiable classes of local and global type 
and allow very general weight systems (so that also Braun--Meise--Taylor classes are covered).
All the ingredients seem to be well-known, even though somewhat scattered in the literature, 
but we think a unified treatment can be useful.

Let us outline the structure of the paper.
In \Cref{sec:interpol}, we recall several interpolation inequalities with short proofs for the convenience of the reader.
Building on these inequalities, we obtain three technical propositions in \Cref{sec:lacest} on which most of the subsequent 
ultradifferentiable regularity results by lacunary estimates are based.
In \Cref{sec:DC}, we present these results for a broad variety of Denjoy--Carleman classes (defined in terms of a single weight sequence)
and, in \Cref{sec:other}, we extend them to more general ultradifferentiable classes (defined by families of weight sequences)
including the Braun--Meise--Taylor classes.
In the last \Cref{sec:compare}, we compare our results with the approach of Albano and Mughetti \cite{Albano:2023aa}
and discuss optimality of the conditions imposed on the base sequence $(k_j)$ and on the weights.
In \Cref{sec:appendix}, we generalize to general weight systems a construction of \cite{Albano:2023aa} which shows that, in general, boundedness of $k_{j+1}/k_j$ 
cannot be omitted.

\section{Interpolation inequalities} \label{sec:interpol}

\subsection{The global setting}

The Landau--Kolmogorov inequality states that a $C^m$ function $f : \R \to \R$ with finite $\|f\|_{L^\infty(\R)}$ and $\|f^{(m)}\|_{L^\infty(\R)}$ 
satisfies 
\[
    \|f^{(j)}\|_{L^\infty(\R)} \le K_{m,j}\, \|f\|_{L^\infty(\R)}^{1-j/m} \|f^{(m)}\|_{L^\infty(\R)}^{j/m}, \quad j=1,\ldots,m-1.
\]
Due to Kolmogorov \cite{MR0001787}, 
the optimal constants $K_{m,j}$ are given by 
\[
    K_{m,j} = \frac{k_{m-j}}{k_m^{1-j/m}}, 
\]
where $k_r := \frac{4}{\pi} \sum_{i=0}^\infty [\frac{(-1)^i}{2i+1}]^{r+1}$ are the Favard constants.
Note that $1 \le k_r\le 2$ so that $K_{m,j} \le 2$.
By a simple functional-analytic argument, Certain and Kurtz \cite{Certain:1977aa} inferred that, if $(E,\|\cdot\|)$ is a real Banach space 
and $A$ is the generator of a strongly continuous group of isometries, then 
\begin{equation*}
   \|A^j x\|\le K_{m,j}\, \|x\|^{1-j/m} \|A^m x\|^{j/m}, \quad j=1,\ldots,m-1, 
\end{equation*}
for $x$ in the domain of $A^m$. 
In particular (see also \cite[Section 4.4]{Kwong:1992}), we have the following lemma.

\begin{lemma} \label[l]{lem:minterpol}
    Let $1 \le p \le \infty$ and $m \in \N_{\ge 2}$.
Let $f : \R^n \to \R$ be a $C^m$ function.  Then, for 
   all $v \in \mathbb S^{n-1}$,
   \begin{equation*}
    \|d_v^jf\|_{L^p(\R^n)} \le 2\, \|f\|_{L^p(\R^n)}^{1-j/m} \|d_v^mf\|_{L^p(\R^n)}^{j/m}, \quad j=1,\ldots,m-1,
   \end{equation*}
   if the right-hand side is finite,
   where $d_v^jf(x) := \p_t^j f(x+tv)|_{t=0}$.
\end{lemma}

\begin{proof}
    We may  assume $v=(1,0,\ldots,0)$ by choosing a suitable orthonormal system of coordinates.
   Let us first consider $p=\infty$.
We have 
\begin{align*}
    \sup_{x \in \R^n} |\p_1^j f(x)| &\le \sup_{\substack{x_i \in \R\\ 2\le i \le n}} \Big[2\,  \Big(\sup_{x_1 \in \R}|f(x)|\Big)^{1-j/m} 
    \Big(\sup_{x_1 \in \R}|\p_1^mf(x)|\Big)^{j/m}\Big]
    \\
                                    &\le 2\,  \Big(\sup_{x \in \R^n}|f(x)|\Big)^{1-j/m} 
    \Big(\sup_{x \in \R^n}|\p_1^mf(x)|\Big)^{j/m}.
\end{align*}
For $p<\infty$, we apply the above remarks to $A = \p_1$ and $E= L^p(\R^n)$. 
It generates the group of translations $T(s)f(x_1,\ldots,x_n) = f(x_1+s,x_2,\ldots,x_n)$ which is strongly 
continuous, by the dominated convergence theorem.
\end{proof}

\begin{remark}
   For $p=2$, the factor $2$ can be omitted as follows from an application of the Fourier transform:
 \begin{align*}
    \int |\xi_1^j \widehat f(\xi)|^2 \, d\xi &=  
   \int (|\xi_1|^{m} |\widehat f(\xi)|)^{\frac{2j}{m}}|\widehat f(\xi)|^{2(1-\frac{j}{m})} \, d\xi   
                                             \le \|\xi_1^m \widehat f(\xi)\|^{2j/m}_{L^2(\R^n)}
                                             \|\widehat f(\xi)\|^{2(1-j/m)}_{L^2(\R^n)},
\end{align*}  
by H\"older's inequality. 
\end{remark}

\subsection{The local setting}

The next two lemmas
follow from an easy adaptation of the proof of \cite[Lemmas 2.3-2.5]{Bolley:1991ab}. 
We sketch the argument for the convenience of the reader.

\begin{lemma} \label[l]{lem:ainterpol}
    There is a constant $C>0$ such that the following holds.
    Let $1 \le p \le \infty$, $a>0$,
    $m \in \N$, $0 \le j \le m$, and $f \in C^m([-2a,2a])$. Then
    \begin{equation*}
        \frac{a^j}{j!}\|f^{(j)}\|_{L^p([-a,a])} \le C^m \Big(\frac{a^m}{m!}\|f^{(m)}\|_{L^p([-2a,2a])} + \|f\|_{L^p([-2a,2a])}\Big).
    \end{equation*}
\end{lemma}

\begin{proof}
    By rescaling, it suffices to assume $a=1$. 
    Fix $m$. Let $\vh$ be a $C^\infty$ function with support in $(-2,2)$ that equals $1$ near $[-1,1]$ and 
    satisfies $\|\vh^{(k)}\|_{L^\infty(\R)} \le (C_0 m)^k$ for $0 \le k \le m$ for a universal constant $C_0$ (cf.\ \cite[Theorem 1.3.5]{Hoermander83I}).
    By Taylor's theorem, for $t \in [-2,2]$ and $0 \le j \le m-1$,
    \begin{align*}
        (\vh f)^{(j)}(t)  = \int_{-2}^t \frac{(t-s)^{m-j-1}}{(m-j-1)!} (\vh f)^{(m)}(s)\, ds
        = \sum_{i=0}^m \binom{m}{i} A_i(t),
    \end{align*}
    where
    \[
        A_i(t) := \frac{1}{(m-j-1)!} \int_{-2}^t (t-s)^{m-j-1} \vh^{(m-i)}(s) f^{(i)}(s) \, ds. 
    \]
    For $i=m$, we have
    \[
        |A_m(t)| \le \frac{1}{(m-j-1)!} \int_{-2}^2 |t-s|^{m-j-1} |\vh(s) f^{(m)}(s)| \, ds, 
    \]
    and thus, by Young's inequality,
    \begin{align*}
        \|A_m\|_{L^p([-1,1])} &\le  \|A_m\|_{L^p([-2,2])} 
        \\
                              &\le \frac{1}{(m-j-1)!} \|\mathbf 1_{[-2,2]}t^{m-j-1} \|_{L^1(\R)} \|\vh f^{(m)}\|_{L^p(\R)}
                              \\
                              &\le \frac{2^{m-j+1}}{(m-j)!}  \|f^{(m)}\|_{L^p([-2,2])}.
    \end{align*}
    For $0 \le i \le m-1$ and $t \in [-1,1]$, integration by parts gives
    \begin{align*}
        A_i(t) &= \frac{(-1)^i}{(m-j-1)!}  \int_{-2}^{-1}\p_s^i [(t-s)^{m-j-1} \vh^{(m-i)}(s)] f(s) \, ds
        \\
               &=  \sum_{\ell=0}^k \binom{i}{\ell}\frac{(-1)^{i+\ell}}{(m-j-1-\ell)!}  \int_{-2}^{-1} (t-s)^{m-j-1-\ell} \vh^{(m-\ell)}(s) f(s) \, ds,
    \end{align*}
    where $k:=  \min\{i,m-j-1\}$. Using Young's inequality as before, we get
    \begin{align*}
        \|A_i\|_{L^p([-1,1])} &\le\sum_{\ell=0}^k \binom{i}{\ell}\frac{2^{m-j-\ell+1}}{(m-j-\ell)!} (C_0 m)^{m-\ell} \|f\|_{L^p([-2,2])}.
    \end{align*}
    Consequently,
    \begin{align*}
        \sum_{i=0}^{m-1} \binom{m}{i} \|A_i\|_{L^p([-1,1])} &\le\sum_{i=0}^{m-1} \sum_{\ell=0}^k\binom{m}{i} \binom{i}{\ell}\frac{2^{m-j-\ell+1}}{(m-j-\ell)!} (C_0 m)^{m-\ell} \|f\|_{L^p([-2,2])}
    \end{align*}
    and the lemma follows if we prove that
    \[
        \sum_{i=0}^{m-1} \sum_{\ell=0}^k\binom{m}{i} \binom{i}{\ell}\frac{2^{m-j-\ell+1}}{(m-j-\ell)!} (C_0 m)^{m-\ell} \le C^m j!
    \]
    for a universal constant $C$. The left-hand side equals 
    \begin{align*}
        \sum_{\ell=0}^{m-j-1} \sum_{i=\ell}^{m-1} \frac{(m-\ell)!}{(m-i)!(i-\ell)!} \frac{m!}{\ell!(m-\ell)!}\frac{2^{m-j-\ell+1}}{(m-j-\ell)!} (C_0 m)^{m-\ell}
    \end{align*}
    which is bounded by
    \begin{align*}
        \MoveEqLeft \sum_{\ell=0}^{m-j-1} 2^{m-\ell} 2^m\frac{2^{m-j-\ell+1}}{(m-j-\ell)!} C_0^{m-\ell} e^m \frac{m!}{\ell!}
        \\
        &\le e^m 2^{2m} 2^{j+1} C_0^j\, j! \sum_{\ell=0}^{m-j} \frac{(m-j)!}{\ell! (m-j-\ell)!} (2^2 C_0)^{m-j-\ell}
    \end{align*}
    which is of the required form.
\end{proof}

\begin{lemma} \label[l]{lem:Uinterpol}
   There exist constants $C_0,C>0$, depending only on the dimension $n$, such that the following holds.
   Let $1 \le p \le \infty$.
   Let $U, V$ bounded open subsets of $\R^n$ such that $\ol U \subseteq V$.
   Let $f \in C^m(\ol V)$.
   For all $0<a \le C_0 \on{dist}(U,\p V)$, all $0 \le j \le m$, and all $v \in \mathbb S^{n-1}$,
    \begin{equation*}
        \frac{a^j}{j!}\|d^j_v f\|_{L^p(U)} \le C^m \Big(\frac{a^m}{m!}\|d^m_v f\|_{L^p(V)} + \|f\|_{L^p(V)}\Big).
    \end{equation*}
\end{lemma}

\begin{proof}
   We may assume that $v=(1,0,\ldots,0)$.
    Let us first show the assertion for $U = (-a,a)^n$ and $V=(-2a,2a)^n$.
   For $x':= (x_2,\ldots,x_n) \in (-a,a)^{n-1}$, \Cref{lem:ainterpol} gives
   \begin{align*}
        \frac{a^j}{j!}\|\p_1^j f(\cdot,x')\|_{L^p((-a,a))} \le C^m \Big(\frac{a^m}{m!}\|\p_1^m f(\cdot,x')\|_{L^p((-2a,2a))} + \|f(\cdot,x')\|_{L^p((-2a,2a))}\Big).
   \end{align*}
   For $p=\infty$, the statement follows, by taking the supremum over all $x' \in (-a,a)^{n-1}$.
   For $p<\infty$, take the $p$-th power an integrate over $x' \in (-a,a)^{n-1}$.

   In general, 
    there is a constant $C_0>0$ depending only on $n$
    such that, if $a \le C_0 \on{dist}(U,\p V)$, 
    then
   $U$ can be covered by a family $\cQ$ of cubes $Q=x+[-a,a]^n$ such that $2Q=x + [-2a,2a]^n \subseteq V$ 
   and such that any two cubes in $\cQ$ have disjoint interior. Then, for each $Q \in \cQ$, 
   \begin{align*}
        \frac{a^j}{j!}\|\p_1^j f\|_{L^p(Q)} \le C^m \Big(\frac{a^m}{m!}\|\p_1^m f\|_{L^p(2Q)} + \|f\|_{L^p(2Q)}\Big).
   \end{align*}
   For $p=\infty$, take the maximum over all $Q\in \cQ$, for $p<\infty$, take the $p$-power and the sum over all $Q \in \cQ$.
\end{proof}

Now we combine \Cref{lem:Uinterpol} with \Cref{lem:minterpol} (following \cite[p.\ 193]{Liess:1990aa}).

\begin{corollary} \label[c]{cor:aminterpol}
    Let $1 \le p \le \infty$.
    Let $U, V$ bounded open subsets of $\R^n$ such that $\ol U \subseteq V$.
    Let $f \in C^\infty(\ol V)$.
    Then, for each integer $m \ge 2$, $j= 1,\ldots,m-1$, and all $v \in \mathbb S^{n-1}$,
    \[
        \|d_v^j f\|_{L^p(U)} \le C^m \|f\|^{1-j/m}_{L^p(V)} \Big(\|d_v^m f\|^{j/m}_{L^p(V)} +  m^j \|f\|^{j/m}_{L^p(V)} \Big),
    \]
    where $C>0$ is a constant depending only on $U$, $V$, and $n$.
\end{corollary}

\begin{proof}
Let $W$ be the open $\de$-neighborhood of $\ol U$, where $\de := \frac{1}{2} \on{dist}(U,\p V)$. 
Let $\vh$ be a $C^\infty$ function with support in $W$ that equals $1$ near $\ol U$ and 
satisfies $\|\p^\al \vh\|_{L^\infty(\R^n)} \le C_0^{|\al|+1} m^{|\al|}$ for $|\al| \le m$, 
where $C_0$ only depends on $U$, $V$, and $n$ (cf.\ \cite[Theorem 1.4.2]{Hoermander83I}). 
    Then $\vh f$ has a $C^\infty$ extension by zero outside $W$ to all of $\R^n$ which we also denote by $\vh f$. 
    By \Cref{lem:minterpol},
\[
    \|d_v^jf\|_{L^p(U)} \le \|d^j_v (\vh f)\|_{L^p(\R^n)}  \le 2\, \|(\vh f)\|^{1-j/m}_{L^p(W)} \|d_v^m(\vh f)\|_{L^p(W)}^{j/m}
\]
and, by \Cref{lem:Uinterpol} with $a=C_1\de$, where $C_1=C_1(n)$,
\begin{align*}
    \|& d_v^m(\vh f)\|_{L^p(W)} 
    \le \sum_{i=0}^m \binom{m}{i} \|d_v^{m-i}\vh\|_{L^\infty(\R^n)} \|d_v^if\|_{L^p(W)}
\\
                               &\le \sum_{i=0}^m \binom{m}{i}  C_0(C_0m)^{m-i} C^m\frac{i!}{(C_1\de)^i} \Big(\frac{(C_1\de)^m}{m!} \|d_v^m f\|_{L^p(V)} 
                               + \|f\|_{L^p(V)}\Big)
                               \\
                               &\le  \sum_{i=0}^m \binom{m}{i} \frac{C_0(e C_0 C \max\{C_1\de,1\})^m}{(C_0C_1\de)^i}\big( \|d_v^mf\|_{L^p(V)}
                                   + m!\, \|f\|_{L^p(V)}\big)
                               \\
                               &\le C_2^m \Big(  \|d_v^mf\|_{L^p(V)} +  m^m \|f\|_{L^p(V)}\Big).
\end{align*}
This implies the assertion.
\end{proof}

\subsection{The Cartan--Gorny inequality}

The following result is due to Gorny \cite{Gorny39} and independently to Cartan \cite{Cartan40}.

\begin{proposition} \label[p]{prop:CartanGorny}
    Let $I \subseteq \R$ be a compact interval and $f \in C^m(I)$.
  Then, for $j=1,\ldots,m-1$, 
  \[
      \|f^{(j)}\|_{L^\infty(I)} \le 4 e^{2j} \Big(\frac{m}{j}\Big)^j \, \|f\|_{L^\infty(I)}^{1- j/m} \max\Big\{\|f^{(m)}\|_{L^\infty(I)}, 
      \frac{m!}{|I|^m} \|f\|_{L^\infty(I)}  \Big\}^{j/m}.
  \]
\end{proposition}

The factor $(\frac{m}{j})^j$ is bounded by $e^m$.

\section{Ultradifferentiable regularity by lacunary estimates} \label{sec:lacest}

In this section, we prove three technical propositions 
on which most of the subsequent results are based.

Let $0=: k_0 < k_1 < k_2 < \cdots $ be a strictly increasing infinite sequence of positive integers.
A sequence $(k_j)$ with these properties is called a \emph{base sequence}.
If additionally $k_{j+1}/k_j$, for $j\ge1$, is bounded we say that $(k_j)$ is a \emph{special base sequence}.

\begin{proposition} \label[p]{prop:keyglobal1}
    Let $(k_j)$ be a base sequence.
    Let $(m_j)$ and $(m'_j)$ be positive increasing sequences of reals satisfying $m_j \le m'_j$ for all $j$.
    Assume that $m_{k_{j+1}}/m'_{k_j}$ is bounded. Let $1 \le p \le \infty$
    and $f \in C^\infty(\R^n)$.

    If there are $C,\rh>0$ such that
    \begin{equation} \label{eq:DC1global1}
   \max_{|\al|=k_j}  \|\p^{\al}f\|_{L^p(\R^n)} \le C (\rh m_{k_j})^{k_j}, \quad j \ge 0,
\end{equation}
then there exist $C',\rh'>0$ such that
\begin{equation} \label{eq:DC2global1}
    \max_{|\al|=\ell}\|\p^\al f\|_{L^p(\R^n)} \le C' (\rh' m'_{\ell})^{\ell}, \quad \ell \ge 0.
\end{equation}
We have $\rh' =  O(\rh)$ as $\rh \to 0$.
\end{proposition}

\begin{proof}
    
    By \eqref{eq:DC1global1},  there are constants $C_1,\rh_1>0$ such that 
    \begin{equation} \label{eq:dirder}
        \|d_v^{k_j}f\|_{L^p(\R^n)} \le C_1 (\rh_1 m_{k_j})^{k_j}, \quad j \ge 0,\, v \in \mathbb S^{n-1}.
\end{equation}
Let $\ell \ge0$ and let $j\ge 0$ be such that $k_j \le \ell < k_{j+1}$.
By \Cref{lem:minterpol} and the assumptions on $(k_j)$, $(m_j)$, and $(m'_j)$, 
\begin{align}
    \label{eq:comput}
    \|d_v^\ell f\|_{L^p(\R^n)} &\le 2\, \|d_v^{k_j}f\|_{L^p(\R^n)}^{1-\frac{\ell -k_j}{k_{j+1}-k_j}} \|d_v^{k_{j+1}}f\|_{L^p(\R^n)}^{\frac{\ell -k_j}{k_{j+1}-k_j}}
                               \\ \nonumber
                               &\le 2C_1\, (\rh_1 m_{k_j})^{k_j\big(1-\frac{\ell -k_j}{k_{j+1}-k_j}\big)}(\rh_1 m_{k_{j+1}})^{k_{j+1}\frac{\ell -k_j}{k_{j+1}-k_j}}
                               \\ \nonumber
                               &\le 2C_1\, (\rh_1 m'_{k_j})^{k_j\big(1-\frac{\ell -k_j}{k_{j+1}-k_j}\big)}(C_2\rh_1 m'_{k_{j}})^{k_{j+1}\frac{\ell -k_j}{k_{j+1}-k_j}}
                               \\ \nonumber
                               &\le 2C_1\, (C_2\rh_1 m'_\ell)^\ell. 
\end{align}

By polarization \cite[Lemma 7.13]{KM97}, for $|\al|=\ell$ and $v_1,\ldots, v_\ell$ the 
list of standard unit vectors 
\[
    \underbrace{e_1,\ldots,e_1}_{\al_1}, \underbrace{e_2, \ldots,e_2}_{\al_2}, \ldots,\underbrace{e_n,\ldots,e_n}_{\al_n},
\]
we have 
\begin{align} \label{eq:polar}
      \p^\al f 
                            &=  \frac{1}{\ell!} \sum_{\ep_1,\ldots,\ep_\ell=0}^1 (-1)^{\ell-\sum \ep_\ell}\Big(\sum \ep_i\Big)^\ell 
                            d^\ell_{\sum \ep_i v_i/\sum \ep_i}f.   
  \end{align}
  Therefore,
\begin{align*}
    \| \p^\al f \|_{L^p(\R^n)} \le  
                            2 C_1 (C_2 \rh_1 m'_\ell)^\ell  \frac{1}{\ell!} \sum_{k=0}^\ell \binom{\ell}{k}k^\ell \le 2C_1 (2 eC_2 \rh_1 m'_\ell)^\ell.
\end{align*}
  Thus \eqref{eq:DC2global1} is proved.
\end{proof}

\begin{proposition} \label[p]{prop:key}
    Let $(k_j)$ be a special base sequence.
    Let $(m_j)$ and $(m'_j)$ be positive increasing sequences of reals satisfying $m_j \le m'_j$ for all $j$
    and such that $j/m_j$ is bounded.
    Assume that $m_{k_{j+1}}/m'_{k_j}$ is bounded. Let $1 \le p \le \infty$.
    Let $U, V$ bounded open subsets of $\R^n$ such that $\ol U \subseteq V$.
    Let $f \in C^\infty(\ol V)$.

    If there are $C,\rh>0$ such that
    \begin{equation} \label{eq:DC1}
   \max_{|\al|=k_j}  \|\p^{\al}f\|_{L^p(V)} \le C (\rh m_{k_j})^{k_j}, \quad j \ge 0,
\end{equation}
then there exist $C',\rh'>0$ such that
\begin{equation} \label{eq:DC2}
    \max_{|\al|=\ell}\|\p^\al f\|_{L^p(U)} \le C' (\rh' m'_{\ell})^{\ell}, \quad \ell \ge 0.
\end{equation}
Under the additional assumption that $j/m_j \to 0$, we have that $\rh' =  O(\rh)$ as $\rh \to 0$. 
\end{proposition}

\begin{proof}
    By \eqref{eq:DC1},  there are constants $C_1,\rh_1>0$ such that 
    \[ 
        \|d_v^{k_j}f\|_{L^p(V)} \le C_1 (\rh_1 m_{k_j})^{k_j}, \quad j \ge0,\, v \in \mathbb S^{n-1}.
\]
Let $\ell \ge 0$ and let $j\ge 0$ be such that $k_j \le \ell < k_{j+1}$.
By \Cref{cor:aminterpol} and the assumptions on $(k_j)$, $(m_j)$, and $(m'_j)$, 
\begin{align*}
    &\|d_v^\ell f\|_{L^p(U)}
    \\
    &\le 
C^{k_{j+1}-k_j} \Big( \|d_v^{k_j} f\|^{1-\frac{\ell -k_j}{k_{j+1}-k_j}}_{L^p(V)} \|d_v^{k_{j+1}}f\|^{\frac{\ell -k_j}{k_{j+1}-k_j}}_{L^p(V)}
    + (k_{j+1}-k_j)^{\ell - k_j}\|d_v^{k_j}f\|_{L^p(V)} \Big) 
    \\
    &\le 
    C_1 C_2^{k_j} \Big(  (\rh_1 m_{k_j})^{k_j\big(1-\frac{\ell -k_j}{k_{j+1}-k_j}\big)} (\rh_1 m_{k_{j+1}})^{k_{j+1}\frac{\ell -k_j}{k_{j+1}-k_j}}
    + (C_3 k_j)^{\ell - k_j} (\rh_1 m_{k_j})^{k_j}\Big)
    \\
    &\le 
    C_1 C_2^{k_j} \Big(  (\rh_1 m'_{k_j})^{k_j\big(1-\frac{\ell -k_j}{k_{j+1}-k_j}\big)} (C_4\rh_1 m'_{k_{j}})^{k_{j+1}\frac{\ell -k_j}{k_{j+1}-k_j}}
    + (C_5 m'_{k_j})^{\ell - k_j} (\rh_1 m'_{k_j})^{k_j}\Big)
    \\
    &\le 
    C_1 C_2^{k_j} \Big(  (C_4\rh_1 m'_{k_{j}})^{\ell}
    + ((C_5+\rh_1) m'_{k_j})^{\ell} \Big)
    \\
    &\le C_1 (C_6 m'_{\ell})^\ell. 
\end{align*}

If we assume that $j/m_j \to 0$ (actually $k_j/m_{k_j}\to 0$ is enough), 
then for each $\ta>0$ there is $j_\ta$ such that $k_j/m_{k_j} \le \ta$ for all $j\ge j_\ta$.
Thus $C_5 = C_3 \rh_1$ provided that $j\ge j_{\rh_1}$ and so 
\[
    \|d_v^\ell f\|_{L^p(U)} \le C_1(C_7 \rh_1 m'_{\ell})^\ell, \quad \ell \ge k_{j_{\rh_1}}.
\]
For $\ell<k_{j_{\rh_1}}$, 
\[
    \|d_v^\ell f\|_{L^p(U)} \le C_1(C_6 m'_{\ell})^\ell \le C_1\max\{\rh_1^{-\ell} : \ell <k_{j_{\rh_1}} \}\cdot (C_6\rh_1 m'_{\ell})^\ell.
\]
Since $\rh_1 = O(\rh)$, we get 
\[
    \|d_v^\ell f\|_{L^p(U)} \le C'(\rh' m'_{\ell})^\ell, \quad \ell \ge 0,
\]
where $\rh' = O(\rh)$ and $C'=C'(\rh')$.

To end the proof, it suffices to apply $\|\cdot\|_{L^p(U)}$ to \eqref{eq:polar} and use the estimate for $\|d_v^\ell f\|_{L^p(U)}$.
\end{proof}

\begin{remark}
    The requirement $k_0=0$ for the base sequence is important in \Cref{prop:keyglobal1}. 
    Without this assumption there is no reason why a function satisfying \eqref{eq:DC1global1} should also 
    fulfill \eqref{eq:DC2global1} for $\ell=0$.
    In the local setting of \Cref{prop:key},
    the assumption $k_0=0$ can be made without loss of generality, by adjusting the constant $C$.

    Note that, in the proof of \Cref{prop:key}, the assumption that $k_{j+1}/k_j$ is bounded is necessary to estimate 
    the factor $C^{k_{j+1}-k_j}$. For the term $(k_{j+1}-k_j)^{\ell - k_j}$ it is enough to use that $j/m_j$ and  $m_{k_{j+1}}/m'_{k_j}$ are bounded.
\end{remark}

\begin{proposition} \label[p]{prop:keyglobal2}
    Let $(k_j)$ be a base sequence.
    Let $(m_j)$ and $(m'_j)$ be positive increasing sequences of reals satisfying $m_j \le m'_j$ for all $j$.
    Assume that $m_{k_{j+1}}/m'_{k_j}$ is bounded. 
    Let $f \in C^\infty(\R^n)$.

    If there are $C,\si>0$ such that
    \begin{align} 
        \label{eq:Sass}
        \max_{|\al|=k_j}  \|x^{\al}f\|_{L^\infty(\R^n)} &\le C (\si m_{k_j})^{k_j}, \quad j \ge 0,
\end{align}
then there exist $C',\si'>0$ such that
\begin{align} 
    \label{eq:Sconcl}
    \max_{|\al|=\ell}\|x^\al f\|_{L^\infty(\R^n)} \le C' (\si' m'_{\ell})^{\ell}, \quad \ell \ge 0.
\end{align}
We have $\si' =  O(\si)$.
\end{proposition}

\begin{proof}
    Recall that 
    \[
    |x| \le n^{1/2-1/p} \, \Big(\sum_{i=1}^n |x_i|^p \Big)^{1/p} \quad \text{ if } p\ge 2.
    \]
    So, if $|\al|=k \ge 2$, then
    \begin{align*}
        |x^\al| \le |x|^k \le n^{(k-2)/2} \sum_{i=1}^n |x_i|^k \le n^{(k-2)/2} \sum_{|\al|=k} |x^\al| \le  n^{(3k-2)/2} \max_{|\al| = k} |x^\al|. 
    \end{align*}
    If $|\al|=1$, then
    \[
        |x^\al| \le |x| \le \sum_{i=1}^n |x_i| \le n \max_{|\al| = 1} |x^\al|. 
    \]
    The assumption \eqref{eq:Sass} gives
    \begin{equation*}
        |x|^{k_j} |f(x)| \le C (n^{3/2} \si \,m_{k_j})^{k_j}, \quad j \ge 0,\, x \in \R^n.
    \end{equation*}
   Then, for $k_j \le \ell < k_{j+1}$,
    \begin{align*}
        |x|^\ell |f(x)|  &= (|x|^{k_j}|f(x)|)^{1- \frac{\ell -k_j}{k_{j+1}-k_j}} (|x|^{k_{j+1}} |f(x)|)^{\frac{\ell -k_j}{k_{j+1}-k_j}}.
    \end{align*}
    So \eqref{eq:Sass} implies \eqref{eq:Sconcl}, by a computation analogous to \eqref{eq:comput}.
\end{proof}

\section{Denjoy--Carleman classes} \label{sec:DC}

Let $M=(M_k)_{k\ge 0}$ be a positive sequence of real numbers. 
We say that $M$ is a \emph{weight sequence} if $M_0=1\le M_1 \le \mu_2 \le \mu_3 \le \cdots$, where $\mu_k:= M_k/M_{k-1}$.
Then also the sequence $m_k:= M_k^{1/k}$ is increasing.

Let $U$ be an open subset of $\R^n$ and $1 \le p \le \infty$. The  \emph{$L^p$-based local Denjoy--Carleman class of Roumieu type} on $U$
is the set $\cE^{\{M\}}_{L^p}(U)$ of all $f \in C^\infty(U)$ such that for all open relatively compact subsets $\Om \Subset U$
there exist constants $C,\rh>0$ such that 
\begin{equation} \label{eq:DC3}
   \max_{|\al|=k}  \|\p^{\al}f\|_{L^p(\Om)} \le C (\rh m_{k})^{k}, \quad k \in \N.
\end{equation}
Analogously, the \emph{$L^p$-based local Denjoy--Carleman class of Beurling type} is the set
$\cE^{(M)}_{L^p}(U)$ of all $f \in C^\infty(U)$ that satisfy \eqref{eq:DC3} for all $\Om$ and all $\rh$ with $C=C(\Om,\rh)$.
We use the placeholder $[\cdot]$ for either the Roumieu case $\{\cdot\}$ or the Beurling case $(\cdot)$.  

Replacing \eqref{eq:DC3} with various global requirements we obtain natural \emph{global Denjoy--Carleman classes},
\begin{align*}
    \cB^{\{M\}}_{L^p}(\R^n)&:= \{f \in C^\infty(\R^n): \exists C,\rh>0 ~\forall \al:   \|\p^{\al}f\|_{L^p(\R^n)} \le C (\rh m_{|\al|})^{|\al|}\}, 
    \\
    \cB^{(M)}_{L^p}(\R^n)&:= \{f \in C^\infty(\R^n): \forall \rh>0~\exists C>0 ~\forall \al:  \|\p^{\al}f\|_{L^p(\R^n)} \le C (\rh m_{|\al|})^{|\al|}\}, 
\end{align*}
\emph{Gelfand--Shilov classes} $\cS^{\{M\}}(\R^n)$ (resp.\ $\cS^{(M)}(\R^n)$) consisting of all $f \in C^\infty(\R^n)$ 
such that there exist $\rh,\si>0$ such that (resp.\ for all $\rh,\si>0$)
\[
    \sup_{\al,\be \in \N^n} \sup_{x \in \R^n} \frac{|x^\al \p^\be f(x)|}{(\rh m_{|\al|})^{|\al|} (\si m_{|\be|})^{|\be|}} < \infty,
\]
and the classes $\cD^{[M]}(\R^n) := \cB^{[M]}_{L^\infty}(\R^n) \cap  C^\infty_c(\R^n)$ of all $\cB^{[M]}_{L^\infty}$-functions with compact support.  
(Cleary, $\cD^{[M]}(\R^n)$ is nontrivial only in the non-quasianalytic setting.)
Note that various aspects of these global classes have recently been studied in \cite{KrieglMichorRainer14a,KrieglMichorRainer16,NenningRainer17}.
We have continuous (with respect to the natural locally convex topologies) inclusions 
    \[
      \xymatrix{
          \cD^{\{M\}}(\R^n) \ar[r]  & \cS^{\{M\}}(\R^n) \ar[r]  & \cB^{\{M\}}_{L^p}(\R^n)
           \ar[r] & \cE^{\{M\}}_{L^p}(\R^n) \\
          \cD^{(M)}(\R^n) \ar[r] \ar[u] & \cS^{(M)}(\R^n) \ar[r] \ar[u] & \cB^{(M)}_{L^p}(\R^n)
          \ar[u] \ar[r] & \cE^{(M)}_{L^p}(\R^n) \ar[u]\\
      }
    \]
    and, for $1 \le p \le q \le \infty$,    
    \begin{equation} \label{eq:p=q}
        \xymatrix{
            \cB^{[M]}_{L^p}(\R^n) \ar[r] & \cB^{[M]}_{L^q}(\R^n)& & \cE^{[M]}_{L^p}(U) \ar@{=}[r] & \cE^{[M]}_{L^q}(U),
        }
    \end{equation}
    provided that $M$ is \emph{derivation closed}, i.e., 
 \begin{equation} \label{eq:dc}
        \sup_{k\ge 0}\Big(\frac{M_{k+1}}{M_k}\Big)^{1/(k+1)} < \infty. 
    \end{equation}
    This follows from the Sobolev inequality, since the latter condition guarantees stability under taking derivatives.
    For the local classes, we have equality if \eqref{eq:dc} holds, because $\|\cdot\|_{L^p(\Om)} \le |\Om|^{1/p-1/q} \|\cdot\|_{L^q(\Om)}$.

If $(k_j)$ is a base sequence, we consider
the sets $\cE^{[M]}_{L^p,(k_j)}(U)$, $\cB^{[M]}_{L^p,(k_j)}(\R^n)$, $\cD^{[M]}_{(k_j)}(\R^n)$
of $C^\infty$ functions that satisfy the defining estimates for multiindices $\al$ with $|\al|=k_j$.

\begin{theorem} \label[t]{thm:DC}
    Let $1 \le p \le \infty$.
    Let $(k_j)$ be a base sequence
    and $M=(M_j)$  a weight sequence such that
    $m_{k_{j+1}}/m_{k_j}$ is bounded.
    We have:
    \begin{enumerate}
        \item[(a)]  $\cB^{[M]}_{L^p,(k_j)}(\R^n)= \cB^{[M]}_{L^p}(\R^n)$.
        \item[(b)]  $\cD^{[M]}_{(k_j)}(\R^n)= \cD^{[M]}(\R^n)$.
    \end{enumerate}
    Let $(k_j)$ be a special base sequence and $U \subseteq \R^n$ open. Then:
    \begin{enumerate}
        \item[(c)] $\cE^{\{M\}}_{L^p,(k_j)}(U)= \cE^{\{M\}}_{L^p}(U)$ provided that $j/m_j$ is bounded. 
        \item[(d)] $\cE^{(M)}_{L^p,(k_j)}(U)= \cE^{(M)}_{L^p}(U)$ provided that $j/m_j \to 0$.
    \end{enumerate}
\end{theorem}

\begin{proof}
    (a) follows from \Cref{prop:keyglobal1} and (b) is a consequence of (a) by intersecting with $C^\infty_c(\R^n)$.
    (c) and (d) follow from \Cref{prop:key}.
\end{proof}

In (c) and (d), the assumption that the base sequence $(k_j)$ is \emph{special}
cannot be omitted; this follows from \Cref{rem:special} below.

That $j/m_j$ is bounded (resp.\ tends to zero) is equivalent to the fact that $\cE^{\{M\}}_{L^\infty}(U)$
(resp.\ $\cE^{(M)}_{L^\infty}(U)$) contains $C^\om(U)$.
Note that $\cE^{[M]}_{L^p}(U) \subseteq\cE^{[M]}_{L^q}(U)$ 
if $1 \le q \le p \le \infty$.

Before we come to a similar result in the Gelfand--Shilov case, let us recall 
a result due to \cite{Chung:1996aa}. (In that paper only the Roumieu case is treated, but the Beurling case can be proved analogously, see 
\cite[Lemma 4]{Rainer:2019ab}.)
To this end we need some further conditions on the weight sequence.
We say that a weight sequence $M=(M_j)$ has \emph{moderate growth} if 
\begin{equation} \label{eq:mg}
    \sup_{j,k\ge1} \Big(\frac{M_{j+k}}{M_jM_k}\Big)^{1/(j+k)} <\infty.
\end{equation}
Note that \eqref{eq:mg} implies \eqref{eq:dc}.

\begin{proposition}[{\cite{Chung:1996aa}}] \label[p]{prop:GSdescript}
  Let $M=(M_j)$ be a weight sequence with moderate growth. 
Assume that $j/m_j$ is bounded in the Roumieu case and tends to zero in the Beurling case.
  Then the following are equivalent:  
\begin{enumerate}
    \item[(a)] $f \in \cS^{[M]}(\R^n)$.
    \item[(b)] There are constants $C,\rh,\si>0$ (resp.\ for all $\rh,\si>0$ there is $C>0$) 
        such that 
        \[
        \sup_x |x^\al f(x)| \le C (\rh m_{|\al|})^{|\al|}\quad \text{ and } \quad \sup_x |\p^\be f(x)| \le C (\si m_{|\be|})^{|\be|}
        \]
        for all $\al,\be \in \N^n$.
    \item[(c)] There are constants $C,\rh,\si>0$ (resp.\ for all $\rh,\si>0$ there is $C>0$) 
        such that 
        \[
        \sup_x |x^\al f(x)| \le C (\rh m_{|\al|})^{|\al|}\quad \text{ and }\quad \sup_x |x^\be \widehat f(x)| \le C (\si m_{|\be|})^{|\be|}
        \]
        for all $\al,\be \in \N^n$.
\end{enumerate}
\end{proposition}

In view of this proposition, we allow a second base sequence $(\ell_j)$ 
and define  
$\cS^{[M]}_{(k_j),(\ell_j)}(\R^n)$ to be the set of $f \in C^\infty(\R^n)$ such that 
there exist $C,\rh,\si>0$ (resp.\ for all $\rh,\si>0$ there is $C>0$) such that
\[
    \max_{|\al|=k_j}  \sup_x |x^\al f(x)| \le C (\rh m_{|\al|})^{|\al|}\quad \text{ and } \quad \max_{|\be|=\ell_j}\sup_x |\p^\be f(x)| \le C (\si m_{\ell_j})^{\ell_j}
        \]
        for all $j \ge 0$.

\begin{theorem}
    Let $(k_j)$ and $(\ell_j)$ be base sequences.
    Let $M=(M_j)$ be a weight sequence with moderate growth.
    Assume that  $m_{k_{j+1}}/m_{k_j}$ and $m_{\ell_{j+1}}/m_{\ell_j}$ are bounded.
    Then:
    \begin{enumerate}
        \item[(a)] $\cS^{\{M\}}_{(k_j),(\ell_j)}(\R^n) = \cS^{\{M\}}(\R^n)$ provided that $j/m_j$ is bounded.
        \item[(b)] $\cS^{(M)}_{(k_j),(\ell_j)}(\R^n) = \cS^{(M)}(\R^n)$ provided that $j/m_j \to 0$.
    \end{enumerate}
\end{theorem}

\begin{proof}
   This follows from \Cref{prop:keyglobal1}, \Cref{prop:keyglobal2}, and \Cref{prop:GSdescript}.
\end{proof}

\begin{remark} \label[r]{rem:Liess}
    Liess \cite{Liess:1990aa} (extending an argument of \cite{Bolley:1991ab}) proved that, under a number of assumptions on the sequence $M=(M_j)$,
    the equality $\cE^{\{M\}}_{L^2,(k_j)}(U)= \cE^{\{M\}}_{L^2}(U)$ implies that $m_{k_{j+1}}/m_{k_j}$ is bounded.
    His assumptions are the following:
    \begin{enumerate}
        \item[(a)] $m_j \le m_{j+1}$ for all $j$ and $j/m_j$ is bounded. 
        \item[(b)] $m_{j+1}/m_j$ is bounded.
        \item[(c)] There exists $g : \N \to (0,\infty)$ such that $m_{kj} \le g(k)m_j$ for all $k$ and $j$.
        \item[(d)] For every $c>0$ there is $c'>0$ such that $m_j \le c \, m_k$ implies $j \le c' k$.  
    \end{enumerate}
    Let us assume that $M$ is a weight sequence of moderate growth such that $j/m_j$ is bounded.
    Then it is easily seen that the conditions (a)--(c) are satisfied.
    On the other hand, for a weight sequence,  (c) is equivalent to \eqref{eq:mg} (take $k=2$ and use \cite[Theorem 1]{Matsumoto84}).

    Observe that, assuming $m_j \le m_{j+1}$, 
    (d) is equivalent to the existence of an integer $n\ge 2$ such that $\liminf_{k \to \infty} m_{nk}/m_k =: a > 1$.
    Indeed, for all $\ell\ge 1$ and large $k$, we have $a^\ell  m_k \le m_{n^\ell k}$. 
    If $a>1$ and $m_j \le c\, m_{k}$, then $c/a^{-\ell}<1$, provided that $\ell$ is large enough, and thus $m_j < m_{n^\ell k}$.
    As $m_j$ is increasing, we conclude that $j \le n^\ell k$. 
It is clear that $a=1$ violates (d).
\end{remark}

\section{Other ultradifferentiable classes} \label{sec:other}

Classically, besides Denjoy--Carleman classes 
a lot of attention was devoted to ultradifferentiable classes, sometimes called Braun--Meise--Taylor classes 
due to the foundational article \cite{BMT90}, 
defined in terms of a weight function; they have their origin in work of Beurling.

A \emph{weight function} is a continuous increasing function $\om : [0,\infty) \to [0,\infty)$ 
satisfying $\om(0)=0$, $\om(2t)=O(\om(t))$, $\log t = o(\om(t))$ as $t \to \infty$ 
and such that $\vh(t) := \om(e^t)$ is convex. 
It is no restriction to assume $\om(t)=0$ for $0 \le t \le 1$ (cf.\ \cite[Section 11.1]{Rainer:2021aa}).

The local classes $\cE^{\{\om\}}_{L^p}(U)$ and $\cE^{(\om)}_{L^p}(U)$ are defined in analogy to Denjoy--Carleman classes, where 
now
\begin{equation}
    \label{eq:DC6}
    \max_{|\al|=k} \|\p^\al f\|_{L^p(\Om)} \le C\, \exp\big( \tfrac{1}{\rh} \vh^*(\rh k)\big), \quad k \ge 0,
\end{equation}
with $\vh^*(s):= \sup_{t\ge 0} (st -\vh(t))$,
plays the role of \eqref{eq:DC3}. More precisely, we have 
\begin{align*}
    \cE^{\{\om\}}_{L^p}(U) &:= \bigcap_{\Om \Subset U} \bigcup_{\rh>0}\bigcup_{C>0} \{ f \in C^\infty(U) :  f \text{ satisfies } \eqref{eq:DC6}\},  
    \\
    \cE^{(\om)}_{L^p}(U) &:= \bigcap_{\Om \Subset U}  \bigcap_{\rh>0}\bigcup_{C>0} \{ f \in C^\infty(U) :  f \text{ satisfies } \eqref{eq:DC6}\}.
\end{align*}
The global classes $\cB^{[\om]}_{L^p}(\R^n)$, $\cS^{[\om]}(\R^n)$, and $\cD^{[\om]}(\R^n)$
are defined in a straightforward way. 
Furthermore, we consider $\cE^{[\om]}_{L^p,(k_j)}(U)$, $\cB^{[\om]}_{L^p,(k_j)}(\R^n)$, $\cS^{[\om]}_{(k_j),(\ell_j)}(\R^n)$, 
and $\cD^{[\om]}_{(k_j)}(\R^n)$, all of them defined in the obvious manner.
(Note that $\cD^{[\om]}(\R^n)$ is nontrivial if and only if $\int_1^\infty t^{-2}\om(t)\,dt <\infty$; cf.\ \cite{BMT90}.)

There is an overarching framework for ultradifferentiable classes introduced in \cite{RainerSchindl12} 
which goes beyond Denjoy--Carleman and Braun--Meise--Taylor classes and, on a technical level, reduces 
the proofs to handling certain families of weight sequences. 
See \cite{Rainer:2021aa} for a comprehensive survey of the theory.

From now on, let $\fM$ be a totally ordered family of weight sequences, i.e., if $M,M' \in \fM$, then either $M_j \le M'_j$ 
or $M'_j \le M_j$ for all $j$.

We define 
\begin{align*}
    \cE^{\{\fM\}}_{L^p}(U) &:= \bigcap_{\Om \Subset U} \bigcup_{M\in \fM} \bigcup_{\rh>0}\bigcup_{C>0} \{ f \in C^\infty(U) :  f \text{ satisfies } \eqref{eq:DC3}\},  
    \\
    \cE^{(\fM)}_{L^p}(U) &:= \bigcap_{\Om \Subset U} \bigcap_{M\in \fM} \bigcap_{\rh>0}\bigcup_{C>0} \{ f \in C^\infty(U) :  f \text{ satisfies } \eqref{eq:DC3}\},  
\end{align*}
and in the evident analogous way also the global classes 
$\cB^{[\fM]}_{L^p}(\R^n)$, $\cS^{[\fM]}(\R^n)$, and $\cD^{[\fM]}(\R^n)$
as well as
$\cE^{[\fM]}_{L^p,(k_j)}(U)$, $\cB^{[\fM]}_{L^p,(k_j)}(\R^n)$, $\cS^{[\fM]}_{(k_j),(\ell_j)}(\R^n)$, 
and $\cD^{[\fM]}_{(k_j)}(\R^n)$.

\begin{theorem} \label[t]{thm:BMT}
Let $1 \le p \le \infty$.
    Let $\om$ be a weight function.
  For  $\fM := \{M^{(\rh)}\}_{\rh>0}$, where $M^{(\rh)}_k := \exp(\tfrac{1}{\rh} \vh^*(\rh k))$, we have the identities: 
    \begin{align*}
        \cE^{[\om]}_{L^p}(U) =  \cE^{[\fM]}_{L^p}(U),
        &\quad \cE^{[\om]}_{L^p,(k_j)}(U) =  \cE^{[\fM]}_{L^p,(k_j)}(U),
        \\
        \cB^{[\om]}_{L^p}(\R^n) =  \cB^{[\fM]}_{L^p}(\R^n),
        &\quad \cB^{[\om]}_{L^p,(k_j)}(\R^n) =  \cB^{[\fM]}_{L^p,(k_j)}(\R^n),
        \\
        \cS^{[\om]}(\R^n) =  \cS^{[\fM]}(\R^n),
        &\quad \cS^{[\om]}_{(k_j),(\ell_j)}(\R^n) =  \cS^{[\fM]}_{(k_j),(\ell_j)}(\R^n),
        \\
        \cD^{[\om]}(\R^n) =  \cD^{[\fM]}(\R^n),
        &\quad \cD^{[\om]}_{(k_j)}(\R^n) =  \cD^{[\fM]}_{(k_j)}(\R^n),
    \end{align*}
    where $(k_j)$ and $(\ell_j)$ are base sequences.
\end{theorem}

\begin{proof}
    This is shown by the proof of \cite[Lemma 5.14]{RainerSchindl12} which is based
    on the following property of $\fM$, see \cite[(5.10)]{RainerSchindl12},
    \begin{equation} \label{eq:BMT}
        \A \si>0 \E H\ge 1 \A \rh>0 \E C\ge 1 \A k \in \N : \si^k M^{(\rh)}_k \le C\, M^{(H\rh)}_k.
    \end{equation}
In fact, for given $\rh,\si>0$
there exist $C,H\ge 1$ such that 
\begin{align*}
    \sup_{j \ge 0}\frac{\max_{|\al|=k_j}\|\p^\al f\|_{L^p(\Om)}}{M^{(H\rh)}_{k_j}} 
    \le C\, \sup_{j \ge 0}\frac{\max_{|\al|=k_j}\|\p^\al f\|_{L^p(\Om)}}{\si^{k_j} M^{(\rh)}_{k_j}}
\end{align*}
and $M^{(H\rh)}_{k_j} = \exp(\tfrac{1}{H\rh} \vh^*(H\rh k_j))$.
In the Gelfand--Shilov case, we observe that additionally 
\begin{align*}
    \sup_{j \ge 0}\frac{\max_{|\al|=\ell_j} \sup_x |x^\al f(x)|}{M^{(H\rh)}_{\ell_j}} 
    \le C\, \sup_{j \ge 0}\frac{\max_{|\al|=\ell_j} \sup_x |x^\al f(x)|}{\si^{\ell_j} M^{(\rh)}_{\ell_j}}.
\end{align*}
The case $\cS^{[\om]}(\R^n) =  \cS^{[\fM]}(\R^n)$ follows from
\begin{align*}
    \sup_{\al,\,\be}\frac{ \sup_x |x^\al \p^\be f(x)|}{M^{(H\rh)}_{|\al|}M^{(H\rh)}_{|\be|}} 
    \le C\, \sup_{\al,\,\be}\frac{ \sup_x |x^\al\p^\be f(x)|}{\ta^{|\al|}\si^{|\be|} M^{(\rh)}_{|\al|}M^{(\rh)}_{|\be|}}.
\end{align*}
taking constants $C,H$ which work for both $\si$ and $\ta$.
\end{proof}

The family $\fM$ from \Cref{thm:BMT} has an inherent moderate growth property (see \cite[(5.6)]{RainerSchindl12}):
\begin{equation} \label{eq:ommg}
    M^{(\rh)}_{j+k} \le M^{(2\rh)}_j M^{(2\rh)}_k, \quad j,k\ge 0.
\end{equation}
In general, we say that $\fM$ has [\emph{moderate growth}] if, in the Roumieu case (i.e., $[\cdot]=\{\cdot\}$),
\[    
\A M \in \fM \E M' \in \fM \E C>0 \A j,k\ge 0 : M_{j+k}\le C^{j+k} M'_j M'_k,  
\]
and, in the Beurling case (i.e., $[\cdot]=(\cdot)$),
\[    
\A M \in \fM \E M' \in \fM \E C>0 \A j,k\ge 0 : M'_{j+k}\le C^{j+k} M_j M_k.
\]

\begin{lemma} \label[l]{lem:mg}
    If $(k_j)$ is a special base sequence and $\fM$ has [moderate growth], then, in the Roumieu case (i.e., $[\cdot]=\{\cdot\}$),
    \begin{equation}   \label{eq:condR} 
    \A M \in \fM \E M' \in \fM : m_{k_{j+1}}/ m'_{k_j} \text{ is bounded},  
\end{equation}
and, in the Beurling case (i.e., $[\cdot]=(\cdot)$),
\begin{equation}  \label{eq:condB}  
\A M \in \fM \E M' \in \fM : m'_{k_{j+1}}/ m_{k_j} \text{ is bounded}.
\end{equation}
\end{lemma}

\begin{proof}
    By assumption, there exists a positive integer $a$ such that $k_{j+1} \le a k_j$ for all $j\ge 1$.
Moderate growth in the Roumieu case guarantees that, for given $M\in \fM$ there exist $C>0$ and $M' \in \fM$ (depending on $a$) such that
\[
    M_{a k_j} \le C^{ak_j} (M'_{k_j})^a.
\]
This implies $m_{k_{j+1}} \le m_{ak_j} \le C m'_{k_j}$.
The Beurling case is similar.
\end{proof}

\begin{theorem} \label[t]{thm:fM}
    Let $1 \le p \le \infty$. 
    Let $(k_j), (\ell_j)$ be base sequences.
    Let $\fM$ satisfy \eqref{eq:condR} in the Roumieu case and \eqref{eq:condB} in the Beurling case. 
    We have:
    \begin{enumerate}
        \item[(a)]  $\cB^{[\fM]}_{L^p,(k_j)}(\R^n)= \cB^{[\fM]}_{L^p}(\R^n)$.
        \item[(b)]  $\cD^{[\fM]}_{(k_j)}(\R^n)= \cD^{[\fM]}(\R^n)$.
    \end{enumerate}
    Assume additionally that, for all $M \in \fM$, $j/m_j$ is bounded, in the Roumieu case, and tends to zero, in the Beurling case.
    Then:
    \begin{enumerate}
        \item[(c)]  $\cE^{[\fM]}_{L^p,(k_j)}(U)= \cE^{[\fM]}_{L^p}(U)$ for all open $U \subseteq \R^n$, provided that $(k_j)$ is a special base sequence.
        \item[(d)]  $\cS^{[\fM]}_{(k_j),(\ell_j)}(\R^n)= \cS^{[\fM]}(\R^n)$, provided that $\fM$ has [moderate growth].
    \end{enumerate}
\end{theorem}

\begin{proof}
    Since the sequences in $\fM$ are totally ordered, we may assume that the sequences provided by \eqref{eq:condR} and \eqref{eq:condB}
    also satisfy $m_j \le m'_j$ or $m'_j \le m_j$ for all $j$, respectively.
Then (a), (b), and (c) follow easily from \Cref{prop:keyglobal1} and \Cref{prop:key}.
Finally, (d) follows from \Cref{prop:keyglobal1}, \Cref{prop:keyglobal2}, and 
the generalization \cite[Lemma 4]{Rainer:2019ab} of \Cref{prop:GSdescript}.
\end{proof}

To deduce a version for Braun--Meise--Taylor classes, it suffices to 
note that for a weight function $\om$ and $\fM = \{M^{(\rh)}\}_{\rh>0}$ the associated family from \Cref{thm:BMT},
\begin{itemize}
    \item $\om(t)=O(t)$ as $t\to \infty$ if and only if $j/m^{(\rh)}_j$ is bounded for some $\rh>0$,
\item $\om(t)=o(t)$ as $t\to \infty$ if and only if $j/m^{(\rh)}_j \to 0$ for all $\rh>0$;
\end{itemize} 
cf.\ \cite[Lemma 5.7 and Corollary 5.15]{RainerSchindl12}.

\begin{theorem} \label[t]{thm:om}
    Let $1 \le p \le \infty$. 
    Let $(k_j)$, $(\ell_j)$ be special base sequences.
    Let $\om$ be a weight function. 
    We have:
    \begin{enumerate}
        \item[(a)]  $\cB^{[\om]}_{L^p,(k_j)}(\R^n)= \cB^{[\om]}_{L^p}(\R^n)$.
        \item[(b)]  $\cD^{[\om]}_{(k_j)}(\R^n)= \cD^{[\om]}(\R^n)$.
    \end{enumerate}
    Assume additionally that $\om(t)=O(t)$ as $t\to\infty$, 
    in the Roumieu case, and $\om(t)=o(t)$, in the Beurling case.
    Then:
    \begin{enumerate}
        \item[(c)]  $\cE^{[\om]}_{L^p,(k_j)}(U)= \cE^{[\om]}_{L^p}(U)$ for all open $U \subseteq \R^n$.
        \item[(d)]  $\cS^{[\om]}_{(k_j),(\ell_j)}(\R^n)= \cS^{[\om]}(\R^n)$.
    \end{enumerate}
\end{theorem}

\begin{proof}
    This is a corollary of \Cref{thm:BMT}, \Cref{lem:mg}, and \Cref{thm:fM}.
\end{proof}

\begin{remark}
    The family $\fM=\{M^{(\rh)}\}_{\rh>0}$ associated with a weight function in \Cref{thm:BMT} 
    has the following property:
    \begin{itemize}
        \item[(a)] Let $0<\si\le\rh$. 
            If there is $C_1>0$ such that $m^{(\si)}_j \le C_1\, m^{(\rh)}_k$, then there is $C_2=C_2(\si,\rh,C_1) >0$
            such that $j \le C_2\, k$.
    \end{itemize}
    So, in this case, validity of \eqref{eq:condR} or \eqref{eq:condB} entails that $(k_j)$ is a special base sequence;
    the converse holds by \eqref{eq:ommg} and \Cref{lem:mg}.

    To see (a), note first that
    for any positive integer $N$ we have
    \[
        m_k^{(N\rh)} = \exp(\tfrac{1}{N\rh k} \vh^*(N\rh k)) = m_{Nk}^{(\rh)}.
    \]
    By \eqref{eq:BMT}, there exist $C=C(\rh)\ge1$ and an integer $H\ge 2$ such that 
    \[
        4 m_k^{(\rh)} \le C^{1/k} m_k^{(H\rh)} = C^{1/k} m_{Hk}^{(\rh)}, \quad k \ge 0.
    \]
    For $k\ge k_0$, we have $C^{1/k}< 2$ and hence 
    $2 m_k^{(\rh)} < m_{Hk}^{(\rh)}$ and, by iteration,
    \[
2^\ell m_k^{(\rh)} < m_{H^\ell k}^{(\rh)}.
    \]
     If $m^{(\si)}_j \le C_1\, m^{(\rh)}_k$, then, choosing $\ell$ such that $C_12^{-\ell} <1$, 
     we find
     \[
         m^{(\si)}_j < m^{(\rh)}_{H^\ell k}.
     \]
     Then the assumption $j \ge \lceil \frac{\rh}{\si} \rceil H^\ell k$ leads to
     \[
         m^{(\rh)}_{H^\ell k} \le m^{(\lceil \frac{\rh}{\si} \rceil\si)}_{ H^\ell k}=m^{(\si)}_{\lceil \frac{\rh}{\si} \rceil H^\ell k} < m^{(\rh)}_{H^\ell k},
     \]
     a contradiction. Thus $j < C_2\, k$, where $C_2 := \lceil \frac{\rh}{\si} \rceil H^\ell$.
For $k < k_0$, 
there are only finitely many $j$ satisfying $m_j^{(\si)} \le C_1 \, m_{k_0}^{(\rh)}$ 
(since $\lim_{j \to \infty}m^{(\si)}_j=\infty$; see \cite[Remark 1.3]{BMT90}).
So adding this bound to $C_2$ gives the assertion for all $k$.
\end{remark}

\begin{remark}
    Under the [moderate growth] assumption, we have, for $1 \le p \le q \le \infty$,  the continuous inclusions
    $\cB^{[\fM]}_{L^p}(\R^n) \subseteq \cB^{[\fM]}_{L^q}(\R^n)$ and $\cE^{[\fM]}_{L^p}(U) = \cE^{[\fM]}_{L^q}(U)$,
    in particular,
    $\cB^{[\om]}_{L^p}(\R^n) \subseteq \cB^{[\om]}_{L^q}(\R^n)$ and $\cE^{[\om]}_{L^p}(U) = \cE^{[\om]}_{L^q}(U)$.
\end{remark}

\section{Comparison with some recent results} \label{sec:compare}

In this section, we compare our result with the ones obtained by Albano and Mughetti \cite{Albano:2023aa} 
and we discuss optimality of the conditions imposed on the base sequence $(k_j)$ and on the weights.

Let $M=(M_j)$ be a weight sequence.
Let $I \subseteq \R$ be a compact interval.
We define 
\begin{align*}
    \cB^{\{M\}}_{L^\infty}(I) &:= \{f \in C^\infty(I) : \E C,\rh>0 \A j \ge 0 : \|f^{(j)}\|_{L^\infty(I)} \le C(\rh m_{j})^j \}, 
    \\
    \cB^{(M)}_{L^\infty}(I) &:= \{f \in C^\infty(I) : \A\rh>0 \E C>0 \A j \ge 0 : \|f^{(j)}\|_{L^\infty(I)} \le C(\rh m_{j})^j \}, 
\end{align*}
as well as $\cB^{[M]}_{L^\infty,(k_j)}(I)$ in the obvious way.

\begin{theorem} \label[t]{thm:CG}
   Let $(k_j)$ be a special base sequence. 
   Let $M=(M_j)$ be a weight sequence such that $m_{k_{j+1}}/m_{k_j}$ is bounded.
   Then: 
    \begin{enumerate}
        \item[(a)] $\cB^{\{M\}}_{L^\infty,(k_j)}(I)= \cB^{\{M\}}_{L^\infty}(I)$ provided that $j/m_j$ is bounded. 
        \item[(b)] $\cB^{(M)}_{L^\infty,(k_j)}(I)= \cB^{(M)}_{L^\infty}(I)$ provided that $j/m_j \to 0$.
    \end{enumerate}
\end{theorem}

\begin{proof}
    Assume that $f \in \cB^{\{M\}}_{L^\infty,(k_j)}(I)$.
    Let $\ell \ge 0$ and let $j\ge 0$ be such that $k_j \le \ell < k_{j+1}$. 
    By \Cref{prop:CartanGorny} and the properties of $(k_j)$ and $(m_j)$,
    \begin{align*}
        &\|f^{(\ell)}\|_{L^\infty(I)} \le 4 e^{2(\ell-k_j)} e^{k_{j+1}-k_j}  \, \|f^{(k_j)}\|_{L^\infty(I)}^{1- \frac{\ell-k_j}{k_{j+1}-k_j}} 
        \\
                    &\hspace{2cm} \cdot \max\Big\{\|f^{(k_{j+1})}\|_{L^\infty(I)}^{\frac{\ell-k_j}{k_{j+1}-k_j}}, 
                                     \Big(\frac{k_{j+1}-k_j}{|I|}\Big)^{\ell-k_j} \|f^{(k_j)}\|_{L^\infty(I)}^{\frac{\ell-k_j}{k_{j+1}-k_j}}\Big\}                                     \\
                                     \\
                                     &\le 4C\,  C_1^\ell  \,  (\rh m_{k_j})^{k_j\big(1- \frac{\ell-k_j}{k_{j+1}-k_j}\big)} 
                                     \Big( (\rh m_{k_{j+1}})^{\frac{k_{j+1}(\ell-k_j)}{k_{j+1}-k_j}} +
                                     \Big(\frac{C_2 k_j}{|I|}\Big)^{\ell-k_j}  (\rh m_{k_j})^{\frac{k_j(\ell-k_j)}{k_{j+1}-k_j}}  \Big)
                                     \\
                                     &\le 4C\,  C_1^\ell  \,  
                                     \Big( (C_3\rh m_{k_{j}})^{\ell} +
                                     \Big(\frac{C_4 m_{k_j}}{|I|}\Big)^{\ell-k_j}  (\rh m_{k_j})^{k_j}  \Big)
                                     \\
                                     &\le 8C\,  (C_5\max\{\rh,C_4 |I|^{-1}\})^\ell  \,  
                                     m_\ell^\ell.
    \end{align*}
    Hence $f \in \cB^{\{M\}}_{L^\infty}(I)$.

    Assume that $f \in \cB^{(M)}_{L^\infty,(k_j)}(I)$.
    Let $\si>0$ be given.
    If we assume that $j/m_j \to 0$, then we may 
    proceed as in the proof of \Cref{prop:key} 
    to see that choosing $\rh>0$ sufficiently small 
    we may find 
    \[
        \|f^{(\ell)}\|_{L^\infty(I)} \le D(\si m_\ell)^\ell, \quad \ell \ge 0,
    \]
    where $D=D(\si)$.
\end{proof}

It is now an easy exercise to deduce the following extensions.

\begin{theorem} 
    Let $(k_j)$ be a special base sequence.
    Let $\fM$ have [moderate growth]
    and assume that, for all $M \in \fM$, $j/m_j$ is bounded, in the Roumieu case, and tends to zero, in the Beurling case.
    Then $\cB^{[\fM]}_{L^\infty,(k_j)}(I)= \cB^{[\fM]}_{L^\infty}(I)$.
\end{theorem}

\begin{theorem} 
    Let $(k_j)$ be a special base sequence.
    Let $\om$ be a weight function such that  $\om(t)=O(t)$ as $t\to\infty$, 
    in the Roumieu case, and $\om(t)=o(t)$, in the Beurling case.
    Then $\cB^{[\om]}_{L^\infty,(k_j)}(I)= \cB^{[\om]}_{L^\infty}(I)$.
\end{theorem}

\begin{remark} \label[r]{rem:Albano}
    In \cite{Albano:2023aa}, 
    instead of assuming $m_{k_{j+1}}/m_{k_j}$ bounded, the authors require that 
    \begin{equation}\label{eq:Albano}
        \E i_0\ge 0 \A i,j >i_0 \text{ and } i<j < ik : m_j \le m_i m_k.
    \end{equation}
    If there is $\ep>0$ such that $m_j>\ep$ for all $j$, then it is not hard to check that \eqref{eq:Albano} implies
    \begin{equation}\label{eq:Albano2}
        \E C\ge 1 \A i,j \ge 1 \text{ and } i<j < ik : m_j \le C\, m_i m_k.
    \end{equation}
    But \eqref{eq:Albano2} (for $k=3$ and $j=2i$) implies that $M$ has moderate growth (cf.\ \cite[Theorem 1]{Matsumoto84}). 
    And, if $(k_j)$ is a special base sequence, then \Cref{lem:mg} implies that $m_{k_{j+1}}/m_{k_j}$ 
    is bounded.
\end{remark}

\begin{remark} \label[r]{rem:special}
    In \Cref{thm:CG}, the assumption that $(k_j)$ is a \emph{special} base sequence 
    cannot be omitted.
    In fact, by \cite[Theorem 1.6]{Albano:2023aa},
    for each positive sequence $M=(M_j)$ such that $\limsup_{j\to \infty} m_j = \infty$
    there is a base sequence $(k_j)$ and $f \in \cB^{\{M\}}_{L^\infty,(k_j)}(I)$ but 
    $f \not\in \cB^{\{M\}}_{L^\infty}(I)$.
    More precisely, a weight sequence $N=(N_j)$ and sequences of positive integers $(k_j)$ and $(\ell_j)$ 
    with the following properties are constructed:
    \begin{itemize}
        \item $\cdots < \ell_j < k_j < \ell_{j+1} < k_{j+1} < \cdots $,
        \item $N_{k_j} = M_{k_j}$, and
        \item $N_{\ell_j} = 2^{2^{\ell_j}} M_{\ell_j}$.
    \end{itemize}
    It is well-known that for each weight sequence $N$ there exists $f \in \cB^{\{N\}}_{L^\infty}(\R)$ with $|f^{(j)}(0)| \ge N_j$ for all $j \ge0$ 
    (cf.\ \cite[Theorem 1]{Thilliez08} or \cite[Lemma 2.4]{Rainer:2021aa})
    from which the assertion follows easily.

    Note that this result does not contradict \Cref{thm:DC}: if also $M=(M_j)$ is a weight sequence,
    then the above conditions imply
    \[
        m_{k_{j+1}}= n_{k_{j+1}}  \ge n_{\ell_{j+1}} = 2^{2^{\ell_{j+1}}/\ell_{j+1}} m_{\ell_{j+1}} \ge 2^{2^{\ell_{j+1}}/\ell_{j+1}} m_{k_j}
    \]
    so that $m_{k_{j+1}}/m_{k_j}$ is unbounded.

    It is possible to adapt the proof of this result 
    to the Beurling case as well as to the general setting of a weight structure $\fM$; 
    see \Cref{sec:appendix}.
\end{remark}

\subsection*{Optimality of the conditions}

Let us summarize the situation in the case of one weight sequence $M=(M_j)$ satisfying that $j/m_j$ is bounded.
Let $I \subseteq \R$ be a compact interval. By \Cref{thm:CG},
\begin{equation} \label{eq:opt}
    \cB^{\{M\}}_{L^\infty,(k_j)}(I)= \cB^{\{M\}}_{L^\infty}(I)
\end{equation}
provided that $k_{j+1}/k_j$ and $m_{k_{j+1}}/m_{k_j}$ are bounded.

As seen in \Cref{rem:Liess} and \Cref{lem:mg}, the dependencies can be summarized by 
the following diagram:
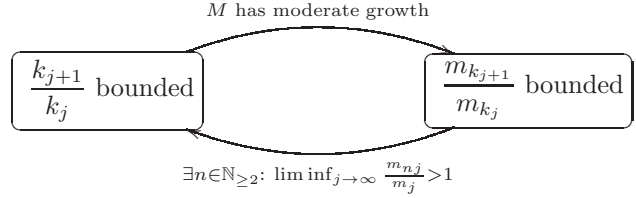
\[
    \xymatrix{
        *+[F-:<3pt>]{\txt{ $\dfrac{k_{j+1}}{k_j}$ bounded }}  \ar@/^2pc/[rrr]^{ M \text{ has moderate growth}}  
        & & &
        *+[F-:<3pt>]{\txt{ $\dfrac{m_{k_{j+1}}}{m_{k_j}}$ bounded }}  \ar@/^2pc/[lll]^{\exists n \in \N_{\ge 2} :~ \liminf_{j\to \infty} \frac{m_{nj}}{m_j}>1}
    }
\]
For instance, for all Gevrey sequences $M=(k^{sk})_k$, where $s \ge 1$,  
$k_{j+1}/k_j$ is bounded if and only if $m_{k_{j+1}}/m_{k_j} = (k_{j+1}/k_j)^s$ is bounded.

Theorem 1.4 in \cite{Albano:2023aa} states that \eqref{eq:opt} holds if $k_{j+1}/k_j$ is bounded and $M$ satisfies \eqref{eq:Albano}. 
In view of \Cref{rem:Albano},
\eqref{eq:Albano} entails moderate growth so that the conditions of \cite[Theorem 1.4]{Albano:2023aa} imply those of \Cref{thm:CG}.

Regarding \Cref{rem:special}, one may ask if $(\ell_j)$ and $(N_j)$ with analogous properties can be found for \emph{each} 
given sequence $(k_j)$ such that $k_{j+1}/k_j$ is unbounded. 
That would imply that boundedness of $k_{j+1}/k_j$ 
is necessary for the validity of \eqref{eq:opt}. 
The construction in \cite{Albano:2023aa} alluded to in \Cref{rem:special} and generalized in \Cref{sec:appendix} 
does not provide this.
On the other hand, in the situation of \Cref{rem:Liess}, in particular, in the Gevrey case, the validity of 
$\cE^{\{M\}}_{L^2,(k_j)}(U)= \cE^{\{M\}}_{L^2}(U)$ implies that $m_{k_{j+1}}/m_{k_j}$ and $k_{j+1}/k_j$ are bounded.

\begin{remark}
    Liess (see \Cref{rem:Liess}) works with $L^2$ based bounds of Roumieu type, but his arguments apply to $L^p$ based bounds for all $1 \le p \le \infty$ 
    of Roumieu and Beurling type
    if the weight sequence $M=(M_j)$ satisfies the properties listed in \Cref{rem:Liess}.  
    (The inclusion $\cE^{[M]}_{L^p,(k_j)}(U) \subseteq \cE^{[M]}_{L^p}(U)$ is continuous, by the closed graph theorem, since 
    convergence in the left-hand side entails pointwise convergence, by the Sobolev inequality.
    In the Beurling case, the quantifiers for the constants in Liess's proof change but this does not affect the conclusion.
    The arguments also work for global classes of type $\cB^{[M]}_{L^p}(I)$, where $I$ is a compact interval.)
\end{remark}

\appendix

\section{} \label{sec:appendix}

Let $\fM = \{M^{(s)} : s>0\}$ be a totally ordered family of weight sequences such that $\lim_{j \to \infty} (M^{(s)}_j)^{1/j} = \infty$ for all $s>0$.
We claim that there exists a weight sequence $N=(N_j)$ and sequences of positive integers $(k_j)$ and $(\ell_j)$ 
    with the following properties:
    \begin{enumerate}
        \item $\cdots < \ell_j < k_j < \ell_{j+1} < k_{j+1} < \cdots $,
        \item $N_{k_j} = j^{-k_j} M^{(1/j)}_{k_j}$ for all $j\ge 1$, and
        \item $N_{\ell_j} = 2^{2^{\ell_j}} M^{(j)}_{\ell_j}$ for all $j\ge 1$.
    \end{enumerate}
There exists $f \in \cB^{\{N\}}_{L^\infty}(\R)$ with $|f^{(j)}(0)| \ge N_j$ for all $j \ge0$
    (cf.\ \cite[Theorem 1]{Thilliez08} or \cite[Lemma 2.4]{Rainer:2021aa}).
Then (2) guarantees that $f \in \cB^{(\fM)}_{L^\infty,(k_j)}(\R)$.
On the other hand, 
$f \not\in \cB^{\{\fM\}}_{L^\infty}(\R)$, by (3). This shows that 
\begin{equation} \label{eq:cex}
    \cB^{(\fM)}_{L^\infty,(k_j)}(\R) \ne \cB^{(\fM)}_{L^\infty}(\R) \quad \text{ and }\quad \cB^{\{\fM\}}_{L^\infty,(k_j)}(\R) \ne \cB^{\{\fM\}}_{L^\infty}(\R).
\end{equation}
It is easy to adjust the arguments so that they give \eqref{eq:cex} for $\R$ replaced with any interval $I \subseteq \R$. 

Let us construct $N=(N_j)$, $(k_j)$, and $(\ell_j)$ with the desired properties.
Let $N_0 := 1$. We will define $N_j$ in terms of $\nu_j := N_j/N_{j-1}$ for $j \ge 1$. 
For simplicity of notation, we put $A_j := 2^{2^j}$. 

Let us first assume that $(k_j)$ and $(\ell_j)$ are arbitrary positive sequences of integers satisfying (1).
We will define $\nu_j$ such that (2) and (3) hold. In the end, we discuss how $(k_j)$ and $(\ell_j)$ must be chosen such that $(\nu_j)$ 
is increasing, i.e., $N=(N_j)$ is a weight sequence. 

For $j \ge 1$, set 
\begin{align*}
    \nu_{k_{j}} &:= \Big( \frac{M^{(1/j)}_{k_{j}}}{j^{k_{j}} A_{\ell_{j}} M^{(j)}_{\ell_{j}}} \Big)^{1/(k_{j}-\ell_{j})},
    \\
    \nu_{\ell_{j+1}} &:= \Big( \frac{j^{k_j} A_{\ell_{j+1}} M^{(j+1)}_{\ell_{j+1}}}{M^{(1/j)}_{k_j}} \Big)^{1/(\ell_{j+1}-k_j)},
    \\
    \nu_{k} &:= \nu_{k_{j}}, \quad \text{ for } \ell_{j}+1 \le k \le k_{j},
    \\
    \nu_{k} &:= \nu_{\ell_{j+1}}, \quad \text{ for } k_{j}+1 \le k \le \ell_{j+1}.
\end{align*}    
Additionally,
\begin{align*}
    \nu_1 = \cdots = \nu_{\ell_1} := (A_{\ell_1} M^{(1)}_{\ell_1})^{1/\ell_1}.
\end{align*}
By construction, (2) and (3) are satisfied.
Taking $\ell_1$ large enough we have $\nu_1\ge 1$ (since $(M^{(1)}_j)^{1/j} \to \infty$).
Next let us check that we may choose $k_1 < \ell_2 < k_3 < \cdots$ in such a way that $(\nu_j)$ is increasing.
More precisely, we have to make sure that 
\[
    \nu_{\ell_j} \le \nu_{k_j} \le \nu_{\ell_{j+1}}\quad \text{ for all } j\ge1.
\]
It is easy to see that taking $k_1 > \ell_1$ sufficiently large gives $\nu_{\ell_1} \le \nu_{k_1}$ (since $(M^{(1)}_j)^{1/j} \to \infty$).
Next $\nu_{k_j} \le \nu_{\ell_{j+1}}$ for $j \ge 1$ amounts to
\begin{align} \label{eq:mon}
    \Big( \frac{M^{(1/j)}_{k_{j}}}{j^{k_{j}} A_{\ell_{j}} M^{(j)}_{\ell_{j}}} \Big)^{1/(k_{j}-\ell_{j})} 
    \le \Big( \frac{j^{k_j} A_{\ell_{j+1}} M^{(j+1)}_{\ell_{j+1}}}{M^{(1/j)}_{k_j}} \Big)^{1/(\ell_{j+1}-k_j)}.
\end{align}
If $\ell_j < k_j$ are already chosen, then this can be achieved by taking $\ell_{j+1}>k_j$ sufficiently large since $(M^{(j+1)}_k)^{1/k} \to \infty$ as $k \to \infty$:
if the left-hand side $L$ of \eqref{eq:mon} satisfies $L\le 1$, then it is clear; 
otherwise $L^{\ell_{j+1}-k_j} \le L^{\ell_{j+1}}$ and $$L \le \Big( \frac{j^{k_j} A_{\ell_{j+1}} M^{(j+1)}_{\ell_{j+1}}}{M^{(1/j)}_{k_j}} \Big)^{1/\ell_{j+1}}$$ 
visibly can be achieved.
Finally, $\nu_{\ell_j} \le \nu_{k_j}$ for $j \ge 2$ means
\begin{align} \label{eq:mon2}
    \Big( \frac{(j-1)^{k_{j-1}} A_{\ell_{j}} M^{(j)}_{\ell_{j}}}{M^{(1/(j-1))}_{k_{j-1}}} \Big)^{1/(\ell_{j}-k_{j-1})}  \le  \Big( \frac{M^{(1/j)}_{k_{j}}}{j^{k_{j}} A_{\ell_{j}} M^{(j)}_{\ell_{j}}} \Big)^{1/(k_{j}-\ell_{j})}.
\end{align}
Here we assume that $k_{j-1}< \ell_j$ are already chosen and we choose $k_j >\ell_j$ sufficiently large such that \eqref{eq:mon2} holds. 
Similarly as for \eqref{eq:mon}, we see that this is possible because $(M^{(1/j)}_k)^{1/k} \to \infty$ as $k \to \infty$.
This ends the construction of $N=(N_j)$, $(k_j)$, and $(\ell_j)$.

\def\cprime{$'$}
\providecommand{\bysame}{\leavevmode\hbox to3em{\hrulefill}\thinspace}
\providecommand{\MR}{\relax\ifhmode\unskip\space\fi MR }
\providecommand{\MRhref}[2]{%
  \href{http://www.ams.org/mathscinet-getitem?mr=#1}{#2}
}
\providecommand{\href}[2]{#2}

\end{document}